\documentclass[12pt]{article}
\usepackage{amsthm,amstext}
\usepackage{amsmath}
\usepackage{graphicx}
\usepackage{amsfonts}
\usepackage{amssymb}

\theoremstyle{plain}

\makeatletter
 \textwidth
14cm \topmargin 0.1cm \textheight 22cm
\newcommand{\singlespacing}{\let\CS=\@currsize\renewcommand{\baselinestretch}{1}\tiny\CS}
\newcommand{\doublespacing}{\let\CS=\@currsize\renewcommand{\baselinestretch}{1}\tiny\CS}
\theoremstyle{plain}
\newtheorem{Theorem}{Theorem}[section]
\theoremstyle{definition}
\newtheorem{Definition}[Theorem]{Definition}

\newtheorem{Corollary}[Theorem]{Corollary}

\newtheorem{Lemma}[Theorem]{Lemma}

\newtheorem{Remark}[Theorem]{Remark}

\newtheorem{Proposition}[Theorem]{Proposition}

\begin{document}
\title{\bf  Fuzzy $h$-ideals of a $\Gamma$-hemiring and its operator hemirings}
\author{ \bf S. K. Sardar$^a$, B. Davvaz$^b$, D. Mandal$^{a,}$\footnote{The research is
funded by CSIR, Govt. of India.} \\ \\
$^a$Department of Mathematics, Jadavpur
University, \\
Kolkata - 700 032, India\\
 sksardarjumath@gmail.com\\
 dmandaljumath@gmail.com
  \\ \\
$^b$Department of Mathematics,
Yazd University, \\
Yazd, Iran\\
 davvaz@yazduni.ac.ir\\ bdavvaz@yahoo.com
}

\date{}
\maketitle
\begin{abstract}
Various correspondence between fuzzy $h$-ideals of a $\Gamma$-hemiring and fuzzy $h$-ideals of its
operator hemirings are established and some of their characterization are given using lattice structure and cartesian
product.\\ \\
\textit{AMS Mathematics Subject Classification (2000):} 16Y60, 16Y99, 03E72.\\ \\
\textit{Keywords and Phrases-}\textit{\ }$\Gamma$-hemiring, cartesian product, fuzzy $h$-ideal ($h$-bi-ideal, $h$-quasi-ideal), operator hemiring.
\end{abstract}
\section{Introduction}
Hemirings \cite{Golan} which provide a common generalization of rings and distributive lattices arise naturally
in such diverse areas of mathematics as combinatorics, functional analysis, graph theory, automata theory,
mathematical modelling and parallel computation systems etc.(for example, see \cite{Golan}, \cite{weinert}). Hemirings
have also been proved to be an important algebraic tool in theoretical computer science, see for instance
\cite{weinert}, for some detail and example.
   Ideals of semiring(hemiring) play a central role in the structure theory and
   useful for many purposes. However they do not in general coincide
   with the usual ring ideals and for this reason, their use is somewhat
    limited in trying to obtain analogues of ring theorems for semiring.
    To solve this problem, Henriksen \cite{Henriksen},
    defined a more restricted class of ideals,
    which are called $k$-ideals. A still more restricted class of
    ideals in hemirings are given by Iizuka \cite{Iizuka}, which are called
     $h$-ideals. LaTorre \cite{LaTorre}, investigated $h$-ideals and $k$-ideals in
      hemirings in an effort to obtain analogues of ring theorems for
      hemiring.
      The theory of $\Gamma$-semiring was introduced by Rao \cite{Rao} as a generalization of semiring.
The notion of  $\Gamma$-semiring theory has been enriched by the
introduction of operator semirings of a $\Gamma$-semiring by Dutta
and Sardar\cite{re:Dutta}. To make operator semirings effective in
the study of $\Gamma$-semirings Dutta et. al. \cite{re:Dutta}
established a correspondence between the ideals of a
$\Gamma$-semiring S and the ideals of the operator semirings of $S$.\\
The concept of fuzzy set was introduced by Zadeh\cite{Zadeh} and
has been applied to many branches of mathematics.  The theory of fuzzy $h$-ideals in hemiring was introduced and studied by Jun et. al. \cite{YBjun}, Zhan et.
al. \cite{Zhan}. As a continuation of this Sardar et al\cite{Sardar} studied those
  properties in $\Gamma$-hemiring in terms of fuzzy $h$-ideals. Recently Ma et. al. \cite{Ma} investigated some properties
  of fuzzy $h$-ideals in $\Gamma$-hemirings. In this paper we establish various correspondence between the fuzzy $h$-ideals
  of a $\Gamma$-hemiring $S$ and the fuzzy $h$-ideals of
the operator hemirings of $S$.

\section{Preliminaries}
We recall the following definition from \cite{Golan}.\\
\indent  A {\it hemiring} (respectively, {\it semiring}) is
a non-empty set $S$ on which operations addition and multiplication
have been defined such that $(S,+)$ is a commutative monoid with
identity 0, $(S,\cdot )$ is a semigroup (respectively, monoid with
identity $1_{S}$), multiplication distributes over addition from
either side, $1_{S}\neq0$ and $0s=0=s0$ for all $s\in S$.

Let $S$ and
$\Gamma$ be two additive commutative semigroups with zero. According to \cite{Sardar}, $S$
is called a {\it $\Gamma$-hemiring} if there exists a mapping $S
\times \Gamma \times S \longrightarrow S $ by $ (a,\alpha,b) \mapsto a\alpha b$
satisfying the following conditions:
\begin{itemize}
\item[(1)]  $(a+b) \alpha c =a \alpha c+ b \alpha c$,
\item[(2)] $ a \alpha
(b+c)= a \alpha b+a \alpha c $,
\item[(3)]  $a (\alpha+\beta)b=a \alpha
b +a \beta b$,
\item[(4)] $ a \alpha (b \beta c)=(a \alpha b) \beta c $,
\item[(5)] $ 0_{S} \alpha a=0_{S} =a \alpha 0_{S},$
\item[(6)] $ a0_{\Gamma}b=0_{S}=b 0_{\Gamma} a$,
\end{itemize}
for all $a,b,c \in S $ and for all $\alpha, \beta \in \Gamma $.

For simplification we write 0 instead of $0_{S}$ and $0_{\Gamma}$.

Let $S$ be the set of all $m \times n$ matrices
over $\mathbf{ Z_{0}^{-}}$ (the set of all non-positive integers)
and $\Gamma$ be the set of all $n \times m$ matrices over
$\mathbf{ Z_{0}^{-}}$. Then $S$ forms a $\Gamma$-hemiring with usual
addition and multiplication of matrices.

Now, we recall the following definitions from \cite{re:Dutta}.

Let $S$ be a $\Gamma$-hemiring and $F$ be the
free additive commutative semigroup generated by $ S \times \Gamma
$. We define a relation $\rho $ on $F$ as follows: $$
\displaystyle{\sum_{i=1}^{m}}(x_{i},\alpha_{i}) \ \rho \
\displaystyle{\sum_{j=1}^{n}}(y_{j}, \beta_{j}) \ {\rm  if \ and \ only \ if} \
\displaystyle{\sum_{i=1}^{m}}x_{i} \alpha_{i} a=
\displaystyle{\sum_{j=1}^{n}}y_{j} \beta_{j} a, $$
 for all $a \in
S$ ($m,n \in \mathbf{Z^{+}}$). Then $\rho$ is a congruence relation on $F$. We denote
the congruence class containing
$\displaystyle{\sum_{i=1}^{m}}(x_{i},\alpha_{i})$ by
$\displaystyle{\sum_{i=1}^{m}}[x_{i},\alpha_{i}]$. Then $F/ \rho$
is an additive commutative semigroup. Now,  $F/ \rho$ forms a
hemiring with the multiplication defined by $$
\displaystyle{\left (\sum^{m}_{i=1}[x_{i},\alpha_{i}]\right )
\left (\sum^{n}_{j=1}[y_{j},\beta_{j}]\right )=\sum_{i,j}[x_{i}\alpha_{i}y_{j},\beta_{j}]}.$$
We denote this hemiring by $L$ and call it the {\it left operator
hemiring} of the $\Gamma$-hemiring $S$.
Dually we define the {\it right operator hemiring} $R$ of the
$\Gamma$-hemiring $S$.
Let $S$ be a $\Gamma$-hemiring and $L$ be the
left operator hemiring and $R$ be the right one. If there exists an
element $\displaystyle{\sum^{m}_{i=1}[e_{i},\delta_{i}]} \in L
({\rm resp.}~\displaystyle{\sum^{n}_{j=1}}[\gamma_{j},f_{j}] \in R)$
such that $\displaystyle{\sum^{m}_{i=1}e_{i}\delta_{i}a=a}$
(respectively, $\displaystyle{\sum^{n}_{j=1}a \gamma_{j}f_{j}=a)}$ for
all $a \in S $, then $S$ is said to have the {\it left unity}
$\displaystyle{\sum^{m}_{i=1}[e_{i},\delta_{i}]}$ (respectively, the {\it right
unity}
$\displaystyle{\sum^{n}_{j=1}[\gamma_{j},f_{j}]})$.

Throughout this paper unless otherwise mentioned for different elements of $L$ (respectively, $R$) we take the same
index say $'i'$ whose range is finite that is from $1$ to $n$, for some positive integer $n$.\\

 Let $S$ be a $\Gamma$-hemiring, $L$ be the
left operator hemiring and $R$ be the right one. If there exists an
element $[e, \delta] \in L~({\rm respectively}, ~[\gamma,f] \in R)$ such that $e
\delta a=a~~({\rm respectively}, ~a \gamma f=a)$ for all $a \in S $,  then $S$ is
said to have the {\it strong left unity} $[e,\delta]$ (respectively,  strong
right unity $[\gamma, f]$)  \cite{Rao}.

Let $S$ be a $\Gamma$-hemiring, $L$ be the
left operator hemiring and $R$ be the right one. Let $ P\subseteq
L~(\subseteq R)$. According to \cite{re:Dutta}, we define $P^{+}=\{a \in S: [a,\Gamma] \subseteq
P\}$ (respectively, $P^{*}=\{a \in S:[\Gamma,a] \subseteq P\}$) and for $Q \subseteq S$,
$$Q^{+'}=\left \{ \displaystyle{\sum_{i=1}^{m}} [x_{i}, \alpha_{i}] \in L
:~\left (\displaystyle{\sum_{i=1}^{m}} ([x_{i}, \alpha_{i}]\right )S\subseteq Q
\right \},$$
where $\left (\displaystyle{\sum_{i=1}^{m}} [x_{i}, \alpha_{i}]\right )S$
denotes the set of all finite sums
$\displaystyle{\sum_{i,k}}x_{i}\alpha_{i}s_{k},~~s_{k} \in S$ and
$$Q^{*'}=\left \{ \displaystyle{\sum_{i=1}^{m}} [\alpha_{i},x_{i}] \in R
:~ S \left (\displaystyle{\sum_{i=1}^{m}} ([ \alpha_{i},x_{i}]\right )\subseteq Q
\right \},$$
where $S\left (\displaystyle{\sum_{i=1}^{m}} [x_{i}, \alpha_{i}]\right )$
denotes the set of all finite sums
$\displaystyle{\sum_{i,k}}s_{k}\alpha_{i}x_{i},~~s_{k} \in S$.

A {\it fuzzy subset} $\mu$ of a non-empty set $S$ is
 a function $\mu:S\longrightarrow[0,1]$.

Let $\mu$ be a non-empty
fuzzy subset of a $\Gamma$-hemiring $S$ (i.e., $\mu(x) \neq 0 $ for
some $x \in S
 $). Then $\mu $ is called a {\it fuzzy left ideal} (respectively, {\it fuzzy right ideal}) of $S$ \cite{Sardar} if
 \begin{itemize}
 \item[(1)] $\mu(x+y) \geq \min [\mu(x), \mu(y)]$,
 \item[(2)] $\mu(x \gamma y) \geq \mu(y) $ (respectively, $\mu(x \gamma y) \geq
  \mu(x)$),
\end{itemize}
for all $ x,y \in S$ and $ \gamma \in \Gamma$.
A {\it fuzzy ideal} of a $\Gamma$-hemiring $S$ is a non-empty fuzzy subset
of $S$ which is a fuzzy left ideal as well as a fuzzy right ideal of
$S$. Note that if $\mu$ is a fuzzy left or right ideal of a
$\Gamma$-hemiring $S$, then $\mu(0)$ $\geq$ $\mu(x)$ for all $x \in S$.\\
A left ideal $A$ of a $\Gamma$-hemiring $S$ is called a {\it left
$h$-ideal} if for any $x, z \in S$ and $a, b \in A$, $$x + a + z = b +
z \Longrightarrow x \in A.$$
A {\it right $h$-ideal} is defined analogously. A fuzzy left ideal $\mu$
of a $\Gamma$-hemiring $S$ is called a {\it fuzzy left $h$-ideal} if for all
$a, b, x, z \in$ S, $$x + a + z = b + z \Longrightarrow \mu(x)\geq
\min \{\mu(a), \mu(b)\}.$$
A {\it fuzzy right $h$-ideal} is defined similarly.
By a {\it fuzzy $h$-ideal} $\mu$, we mean that $\mu$ is both fuzzy left
and
fuzzy right $h$-ideal.

For example, let $S$ be the
additive commutative semigroup of all non-positive integers and
$\Gamma$ be the additive commutative semigroup of all non-positive
even integers. Then $S$ is a $\Gamma$-hemiring if $a \gamma b$
denotes the usual multiplication of integers $ a, \gamma, b $
where $ a, b \in S $ and $ \gamma \in \Gamma $. Let $\mu$ be a
fuzzy subset of $S$, defined as follows
$$ \mu(x)= \left
\{\begin{array}{ll}
 1 & {\rm if}\ x=0 \\ 0.7  & {\rm if} \  x
\ {\rm is \ even}
\\  0.1 & {\rm if} \   x\  { \rm is \ odd}
\end{array} \right . $$
The fuzzy subset $\mu$ of $S$ is both a fuzzy ideal and a fuzzy $h$-ideal of $S$.

Let $S$ be a $ \Gamma $-hemiring and $
\mu_{1}, \mu_{2}$ be two fuzzy subsets of $S$ . Then the
sum $ \mu_{1}\oplus \mu_{2} $ is defined as follows:
$$ (\mu_{1} \oplus \mu_{2})(x)= \left \{
\begin{array}{l}
\displaystyle{\sup_{x=u+v}} \{ \min \{
\mu_{1}(u), \mu_{2}(v) \}\ :\ u, v \in S\}  \\
 \ \  0 \ \  {\rm if} \  x \ {\rm can \ not \ be \ expressed \ as \ above}.
 \end{array}
 \right.
 $$

Let $\mu $ and $\theta$ be two fuzzy subsets of
a $\Gamma$-hemiring $S$. We define {\it generalized $h$-product} of $\mu $ and $\theta$ by
$$
\mu
\circ_{h}\theta(x)=\left \{
\begin{array}{l}
\sup\left \{\underset{i}{\min}\{\mu(a_{i}),\mu(c_{i}),\theta(b_{i}),
\theta(d_{i})\}\ :x+\displaystyle{\sum_{i=1}^{n}a_{i}\gamma_{i}b_{i}+z=
\sum_{i=1}^{n}c_{i}\delta_{i}d_{i}+z} \right \}\\
 0 \ \  {\rm if} \  x \ {\rm can \ not \ be \ expressed \ as \ above},
 \end{array}
 \right.
 $$
where $x,z,a_{i},b_{i},c_{i},d_{i}\in S$ and
$\gamma_{i},\delta_{i}\in \Gamma,$ for $i=1,\ldots ,n$.\\
Ma et. al. \cite{Ma} also defined simple $h$-product by
$$
\mu
\Gamma_{h}\theta(x)=\left \{
\begin{array}{l}
\sup \left \{\min
\{\mu(a),\mu(c),\theta(b),
\theta(d)\} \ : x+a\gamma b+z=
c\delta d+z\right \}\\
 0 \ \  {\rm if} \  x \ {\rm can \ not \ be \ expressed \ as \ above},
 \end{array}
 \right.
 $$
where $x,z,a,,c,d\in S$ and
$\gamma,\delta\in \Gamma.$\\
We now recall following two definitions from \cite{Ma}\\
A fuzzy left(right) $h$-ideal $\zeta$ of a
$\Gamma$-hemiring S is said to be {\it prime} if $\zeta$ is a non-constant
function and for any two fuzzy left(right) $h$-ideals $\mu$ and $\nu$
of S, $\mu\Gamma_{h}\nu\subseteq \zeta$ implies $\mu\subseteq\zeta$
 or $\nu\subseteq \zeta.$

Similarly, we can define {\it semiprime} fuzzy $h$-ideal.

A fuzzy subset $\mu$ of a $\Gamma$-hemiring $S$ is
 called {\it fuzzy $h$-bi-ideal} if for all $x,y,z,a,b\in$S and
 $\alpha,\beta\in\Gamma$ we have
 \begin{enumerate}
  \item $\mu(x+y)\geq\min\{\mu(x),\mu(y)\}$,
  \item $\mu(x\alpha y)\geq\min\{\mu(x),\mu(y)\}$,
  \item $\mu(x\alpha y\beta z)\geq\min\{\mu(x),\mu(z)\}$,
  \item $x+a+z=b+z \ \Rightarrow\ \mu(x)\geq\min\{\mu(a),\mu(b)\}$.
\end{enumerate}
A fuzzy subset $\mu$ of a $\Gamma$-hemiring $S$ is
 called {\it fuzzy $h$-quasi-ideal} if for all $x,y,z,a,b\in S$ we have
 \begin{enumerate}
   \item $\mu(x+y)\geq\min\{\mu(x),\mu(y)\}$,
   \item $(\mu o_{h}\chi_{S})\cap(\chi_{S} o_{h}\mu)\subseteq\mu$,
   \item $x+a+z=b+z \ \Rightarrow \ \mu(x)\geq\min\{\mu(a),\mu(b)\}$.
 \end{enumerate}

For more preliminaries of semirings (hemirings) and
$\Gamma$-semirings we refer to \cite{Golan}
and\cite{re:Dutta}, respectively. Also,
for more results on fuzzy $h$-ideals in $\Gamma$-hemirings we refer to \cite{Sardar}.
 Throughout this paper unless otherwise mentioned S denotes a $\Gamma$-hemiring with left
unity and right unity and  FLh-I$(S)$, FRh-I$(S)$ and Fh-I$(S)$ denote
respectively the set of all fuzzy left $h$-ideals, the set of all
fuzzy right $h$-ideals and the set of all fuzzy $h$-ideals of the $
\Gamma $-hemiring $S$. Similar is the meaning of FLh-I$(L)$, FLh-I$(R)$,
FRh-I$(L)$, FRh-I$(R)$, Fh-I$(L)$, Fh-I$(R)$, where L and R are
respectively the left operator and right operator hemirings of the
$\Gamma$-hemiring $S$. Also, in this section we assume that
$\mu(0)=1$ for a fuzzy left $h$-ideal (respectively, fuzzy right $h$-ideal,
fuzzy $h$-ideal) $\mu$ of a $\Gamma$-hemiring $S$. Similarly, we assume
that $\mu(0_{L})=1$ (respectively,  $\mu(0_{R})=1$) for a fuzzy left
$h$-ideal (respectively, fuzzy right $h$-ideal, fuzzy $h$-ideal) $\mu$ of the
left operator hemiring (respectively, right operator hemiring $R$) of a
$\Gamma$-hemiring $S$.
\section{Correspondence of Fuzzy $h$-ideals}
 Throughout this section $S$ denotes a $\Gamma$-hemiring, $R$ denotes the
  right operator hemiring and
  $L$ denotes the left operator
hemiring of the $\Gamma$-hemiring $S$.
Now, we recall the following definitions from \cite{SardarGo}.
\begin{Definition} Let $ \mu$ be a fuzzy subset of $L$. We
define a fuzzy subset $\mu^{+}$ of $S$ by
$$ \mu^{+}(x)=
\displaystyle {\inf_{\gamma \in \Gamma}} \left \{ \mu \left ([x, \gamma] \right )\right \} ,$$ where
$x \in S $.
If $\sigma$ is a fuzzy subset of $S$, we define a fuzzy
subset $\sigma^{+'}$ of $L$ by $$
\sigma^{+'}\left (\displaystyle{\sum_{i}}[x_{i}, \alpha_{i}]\right )
=\displaystyle{\inf_{s \in S}} \left \{ \sigma \left (\displaystyle{\sum_{i}}
x_{i} \alpha_{i}s\right )\right \},$$ where $ \displaystyle{\sum_{i}} [x_{i},
\alpha_{i}] \in L $. \label{Def:3.1}
\end{Definition}

\begin{Definition} Let $\delta$ be a fuzzy subset of $R$. We
define a fuzzy subset $\delta^{*}$ of $S$ by
$$ \delta^{*}(x)
=\displaystyle{\inf_{\gamma \in \Gamma}} \left \{ \delta([ \gamma, x])\right \},$$
where $ x \in S $.
If $\eta$ is a fuzzy subset of $S$, we define a fuzzy subset
$\eta{*'}$ of $R$ by
$$
\eta^{*'}\left (\displaystyle{\sum_{i}}[\alpha_{i},
x_{i}]\right )=\displaystyle \inf_{s \in S} \left \{ \eta \left (\sum_{i} s
\alpha_{i}x_{i}\right )\right \},$$ where $\displaystyle{\sum_{i}} [ \alpha_{i},
x_{i}] \in R $.\label{Def:3.2}
\end{Definition}
\begin{Lemma} If $\{ \mu_{i}: i \in I \}$ is a collection of
fuzzy subsets of $L$, then $\displaystyle{ \bigcap_{i \in I}
\mu_{i}^{+}=\left (\bigcap_{i \in I}
\mu_{i}\right )^{+}}$.\label{Lemma:3.3}
\end{Lemma}
\begin{proof} It is straightforward.
\end{proof}
\begin{Proposition} If $\mu \in$ Fh-I$(L)$, then $ \mu^{+} \in$
Fh-I$(S)$.\label{Prop:3.4}
\end{Proposition}
\begin{proof} Let $\mu \in$ Fh-I$(L)$. Then $\mu(0_{L})=1$.
Since for all $\gamma \in \Gamma ,~[0_{S}, \gamma]
$ is the zero element of $L$, we have
 $\mu^{+}(0_{S})= \displaystyle{\inf_{\gamma \in \Gamma} \left \{\mu([0_{S},
\gamma])\right \}} =1$.
So $\mu^{+}$ is non empty and $\mu^{+}(0_{S})=1$.
Let $x, y \in S$ and $\alpha \in \Gamma$.
Then
$$
\begin{array}{ll}
 \mu^{+}(x+y)& =\displaystyle{\inf_{\gamma \in \Gamma}}\left \{\mu([x+y, \gamma])\right \}\\
 &=\displaystyle{\inf_{\gamma \in
\Gamma}}\left \{\mu([x, \gamma]+[y, \gamma])\right \}\\
&
\geq\displaystyle{\inf_{\gamma \in \Gamma}}\left \{\min \{\mu([x, \gamma]),\mu([y,\gamma])\} \right \}\\
&=\min\left \{\displaystyle{\inf_{\gamma
\in \Gamma}} \left \{\mu([x, \gamma])\right \}, \displaystyle{\inf_{\gamma \in
\Gamma}}\left \{\mu([y, \gamma]) \right \} \right \}
\\
& =\min \{\mu^{+}(x), \mu^{+}(y)\}.
\end{array}
$$
Therefore, $\mu^{+}(x+y) \geq \min \{\mu^{+}(x), \mu^{+}(y)\}$.
Also, we have
$$
\mu^{+}(x \alpha y)  =\displaystyle{\inf_{\gamma \in \Gamma}}
\{\mu([x \alpha y, \gamma])\}
\geq \displaystyle{\inf_{\gamma \in \Gamma}}\{\mu([x, \alpha][y,
\gamma])\}\geq \underset{\gamma\in\Gamma}{\inf}
\{\mu[y, \gamma]\}=\mu^{+}(y)
$$ and
$$\mu^{+}(x \alpha y)=\displaystyle{\inf_{\gamma \in
\Gamma}\{\mu([x \alpha y,\gamma])\}}= \displaystyle{\inf_{\gamma \in \Gamma}}\{\mu([x,
\alpha][y, \gamma])\}
\geq\mu([x, \alpha]) \geq
\displaystyle{\inf_{\delta \in \Gamma}}\{\mu([x,
\delta])\}=\mu^{+}(x).$$
Hence, $\mu^{+}$ is a fuzzy ideal of $S$.
Now, suppose that $x+a+z=b+z$, for $x,a,b,z\in S$. Then
$$
\begin{array}{ll}
\mu^{+}(x) &=\displaystyle{\inf_{\gamma \in \Gamma}\{\mu([x,
\gamma])\}}\\
&\geq\displaystyle{ \inf_{\gamma \in \Gamma} \{\min\{\mu([a,
\gamma]),\mu([b,\gamma])\}\}}\\
&=\min\left \{ \displaystyle{\inf_{\gamma \in \Gamma}\left \{\mu \left ([a,
\gamma] \right ) \right \},\inf_{\gamma \in \Gamma}\left \{\mu([b,\gamma])\right \} } \right \}\\
&=\min\{\mu^{+}(a),\mu^{+}(b)\}.
\end{array}
$$
Therefore, $\mu^{+}$ is a fuzzy $h$-ideal of $S$.
\end{proof}
\begin{Proposition}\label{$h$-ideal+} If $\sigma \in$ Fh-I$(S)$ (respectively,  FRh-I$(S)$,
FLh-I$(S)$), then $\sigma^{+'} \in $ Fh-I$(L)$ (respectively, FRh-I$(L)$,
FLh-I$(L)$).\label{Prop:3.5}\end{Proposition}

\begin{proof} Let $\sigma \in$ Fh-I$(S)$. Then $\sigma(0_{S})=1$.
Now, we have
$$\sigma^{+'}([0_{S}, \gamma])= \displaystyle{\inf_{s \in
S}\{\sigma(0_{S} \gamma s)\} =\inf_{s \in S}\{
\sigma(0_{S})\}}=1,$$
for all $\gamma \in \Gamma$.
Therefore, $\sigma^{+'}$ is non-empty and $\sigma^{+'}(0_{L})=1$ as $[0_{S},\gamma]$ is the zero element of $L$.
Let $\displaystyle{\sum_{i} [x_{i}, \alpha_{i}],~~\sum_{j} [y_{j}, \beta_{j}]} \in L
$. Then
$$
\begin{array}{l}
\sigma^{+'}\displaystyle{\left (\sum_{i}[x_{i}, \alpha_{i}]+ \sum_{j}[y_{j},
\beta_{j}]\right )}\\
=\displaystyle{\inf_{s \in S}\left \{\sigma \left (\sum_{i} x_{i} \alpha_{i}s
+\sum_{j}y_{j} \beta_{j}s \right )\right \}}\\
\geq \displaystyle{\inf_{s \in S}\left \{\min \left \{\sigma \left (\sum_{i} x_{i} \alpha_{i}s \right ),
 \sigma \left (\sum_{j} y_{j} \beta_{j}s \right )\right \} \right \}}
\\
=\min \left \{\displaystyle{\inf_{s \in S}\left \{\sigma \left (\sum_{i} x_{i} \alpha_{i}s \right )\right \},
\inf_{s \in S}\left \{\sigma \left (\sum_{j} y_{j} \beta_{j}
s \right )\right \}} \right \}\\
=\min\left \{ \sigma^{+'}\left (\displaystyle \sum_{i}[x_{i},
\alpha_{i}]\right ), \sigma^{+'} \left (\displaystyle \sum_{j}[y_{j},
\beta_{j}]\right ) \right \}.
\end{array}
$$
Also, we have
$$
\begin{array}{l}
\sigma^{+'}\left (\displaystyle{\sum_{i}[ x_{i}, \alpha_{i}] \sum_{j}[ y_{j}, \beta_{j}]} \right )
\\
=\sigma^{+'}\displaystyle{\left (\sum_{i,j}[ x_{i} \alpha_{i} y_{j},
\beta_{j}]\right )}\\

=\displaystyle{\inf_{s \in S}} \left \{ \sigma \left (\displaystyle{\sum_{i,j}} x_{i}
\alpha_{i} y_{j} \beta_{j} s \right ) \right \}\\
\geq \displaystyle{\inf_{s \in S}
 \left \{\min \left \{ \sigma \left (\sum_{i} x_{i} \alpha_{i}y_{1} \right ),
\sigma \left (\sum_{i}~~ x_{i} \alpha_{i}y_{2} \right ), \sigma \left (\sum_{i} x_{i}
\alpha_{i}y_{3} \right ),\ldots \right \} \right \}}\\
\geq
\displaystyle\min \left \{ \sigma \left (\sum_{i}~~ x_{i} \alpha_{i}y_{1} \right ),
\sigma \left (\sum_{i}~~ x_{i} \alpha_{i}y_{2} \right ), \sigma \left (\sum_{i} x_{i}
\alpha_{i}y_{3} \right ),\ldots \right \}\\
\geq
\displaystyle{\inf_{s \in S}\left \{ \sigma \left (\sum_{i}(x_{i} \alpha_{i}
s)\right )\right \}}\\
=\sigma^{+'}\left (\displaystyle{\sum_{i}}[x_{i}, \alpha_{i}]\right ).
\end{array}
$$
Similarly, we can show that $$\displaystyle{\sigma^{+'} \left (\sum_{i}[ x_{i}, \alpha_{i}] \sum_{j}[ y_{j},
\beta_{j}] \right ) \geq
\sigma^{+'}\left (\sum_{j}[y_{j},\beta_{j}]\right )}.$$
Thus, $\sigma^{+'}$ is a fuzzy ideal of $L$.
Now, suppose that $$\displaystyle{\sum_{i}} [x_{i},
e_{i}]+\displaystyle{\sum_{i}} [a_{i},
\alpha_{i}]+\displaystyle{\sum_{i}} [z_{i},
\delta_{i}]=\displaystyle{\sum_{i}} [b_{i},
\beta_{i}]+\displaystyle{\sum_{i}} [z_{i}, \delta_{i}],$$ where $
\displaystyle{\sum_{i}} [x_{i}, e_{i}],\displaystyle{\sum_{i}}
[a_{i}, \alpha_{i}], \displaystyle{\sum_{i}} [z_{i},
\delta_{i}], \displaystyle{\sum_{i}} [b_{i}, \beta_{i}]\in L$.
Then
$$
\begin{array}{ll}
\sigma^{+'}\displaystyle{ \left (\sum_{i}[x_{i}, e_{i}]\right )}
&=\displaystyle \inf_{s \in S}\left \{\sigma \left (\sum_{i} x_{i}
e_{i}s \right ) \right \} \\
&\geq \displaystyle \inf_{s \in S} \left \{\min \left \{\sigma \left (\sum_{i} a_{i} \alpha_{i}s \right ),
 \sigma \left (\sum_{j} b_{i} \beta_{i}s \right ) \right \} \right \}
\\
&=\min \left \{\displaystyle\inf_{s \in S} \left \{\sigma \left (\sum_{i} a_{i} \alpha_{i}s \right ) \right \},
\inf_{s \in S}\left \{\sigma \left (\sum_{j} b_{i} \beta_{i}
s \right )\right \} \right \}\\
&=\min \left \{ \sigma^{+'} \left (\displaystyle \sum_{i}[a_{i},
\alpha_{i}] \right ), \sigma^{+'} \left (\displaystyle \sum_{j}[b_{i},
\beta_{i}] \right ) \right \}.
\end{array}
$$
Therefore, $\sigma^{+'}$ is a fuzzy $h$-ideal of $L$.
\end{proof}
Similarly, we can prove the following propositions.

\begin{Proposition} If $\delta \in$ Fh-I$(R)$ (respectively, FRh-I$(R)$,
FLh-I$(R)$), then
$\delta^{*} \in$ Fh-I$(S)$ (respectively,  FRh-I$(S)$,
FLh-I$(S)$).\label{Prop:3.6}\end{Proposition}

\begin{Proposition} If $\eta \in $ Fh-I$(S)$ (respectively, FRh-I$(S)$, FLh-I$(S)$),
then
$\eta^{*'} \in $ Fh-I$(R)$ (respectively, FRh-I$(R)$,
FLh-I$(R)$).\label{Prop:3.7}\end{Proposition}

\begin{Theorem} The lattice of all fuzzy $h$-ideals of $S$ and the lattice of all fuzzy $h$-ideals of $L$ are isomorphic via the mapping $\sigma \mapsto
\sigma^{+'}$, where $\sigma \in$ Fh-I$(S)$ and $\sigma^{+'} \in$ Fh-I$(L)$.\label{Th:3.8}\end{Theorem}

\begin{proof} First, we  show that $(\sigma^{+'})^{+}=\sigma
$, where $\sigma \in$ Fh-I$(S)$. Let $x \in S $. Then
$$\begin{array}{l}
((\sigma^{+'})^{+})(x)= \displaystyle{\inf_{\gamma \in \Gamma}
\{\sigma^{+'}([x, \gamma])\}=\inf_{\gamma \in \Gamma}\left \{ \inf_{s \in S
}\{\sigma(x \gamma s)\} \right \}}
\geq
\displaystyle{\inf_{\gamma \in \Gamma}\left \{\inf_{s \in
S}\{\sigma(x)\} \right \}}=\sigma(x) .
\end{array}
$$
So $\sigma \subseteq (\sigma^{+'})^{+}$.
Now, let $\displaystyle{\sum_{i}} [\gamma_{i}, f_{i}]$ be the right unity of $S$. Then $\displaystyle{\sum_{i}} x \gamma_{i}
f_{i}=x $ for all $x \in S$. We have
$$\sigma(x)=\sigma\left (\displaystyle{\sum_{i}} x \gamma_{i} f_{i}\right )
\geq \min \{\sigma(x \gamma_{1} f_{1}),\sigma(x \gamma_{2}
f_{2}),\ldots \}
\geq \displaystyle{\inf_{\gamma
\in \Gamma} \left \{ \inf_{s \in S} \{\sigma(x \gamma s)\} \right \}}=
(\sigma^{+'})^{+}(x).$$
Therefore, $(\sigma^{+'})^{+} \subseteq \sigma $ and so $ (\sigma^{+'})^{+}=\sigma
$. Now, let $\mu \in$ Fh-I$(L)$. Then
$$
\begin{array}{ll}
 ((\mu^{+})^{+'})\left (\displaystyle{\sum_{i}}[x_{i}, \alpha_{i}]\right )
 &=\displaystyle{\inf_{s \in S}\left \{\mu^{+}\left (\sum_{i} x_{i} \alpha_{i}
s \right ) \right \}}\\
&=\displaystyle{\inf_{s \in S} \left \{\inf_{\gamma \in \Gamma}\left \{\mu \left (\sum_{i} [x_{i}
 \alpha_{i} s, \gamma] \right ) \right \} \right \}}\\
&=\displaystyle{\inf_{s \in S} \left \{\inf_{\gamma \in \Gamma}\left \{\mu \left (\sum_{i}
[ x_{i}, \alpha_{i}][ s, \gamma] \right ) \right \} \right \}}\\
& \geq \displaystyle  \inf_{ s \in S}\left \{\inf_{\gamma \in \Gamma}
\left \{\mu \left (\displaystyle \sum_{i} [x_{i}, \alpha_{i}] \right ) \right \} \right \}\\
 &=\mu \left (\displaystyle \sum_{i}[x_{i},
\alpha_{i}] \right ).
\end{array}
$$
So $\mu \subseteq (\mu^{+})^{+'}$.
Let $\displaystyle{\sum_{i}}[e_{i}, \delta_{i}]$ be the left unity of $S$. Then
$$
\begin{array}{ll}
\displaystyle \mu \left (\sum_{j}[x_{j}, \alpha_{j}] \right ) & = \mu \left (\displaystyle \sum_{j}[x_{j}, \alpha_{j}]\displaystyle \sum_{i}[e_{i}, \delta_{i}]
\right )\\
& \geq \displaystyle{\min \left \{\mu \left (\sum_{j}[x_{j}, \alpha_{j}][e_{1}, \delta_{1}] \right ),
\mu \left (\sum_{j}[x_{j}, \alpha_{j}][e_{2},
\delta_{2}] \right ),\ldots \right \}}\\
&\geq \displaystyle\inf_{ s
\in S}\left \{\inf_{\gamma \in \Gamma} \left \{\mu \left (\sum_{j}[x_{j}, \alpha_{j}][s,
\gamma] \right ) \right \} \right \}\\
&=(\mu^{+})^{+'} \left (\displaystyle{\sum_{j}}[x_{j},
\alpha_{j}] \right ).
\end{array}
$$
Thus, $(\mu^{+})^{+'} \subseteq \mu $ and so $(\mu^{+})^{+'} = \mu
$. Therefore, the correspondence $\sigma \mapsto \sigma^{+'}$ is a
bijection.

Now, let $ \sigma_{1}, \sigma_{2} \in$ Fh-I$(S)$ be such that $\sigma_{1} \subseteq
\sigma_{2}$. Then
$$\displaystyle{\sigma_{1}^{+'}\left (\sum_{i}[x_{i},
\alpha_{i}]\right )}
= \displaystyle{\inf_{s \in S}\left \{\sigma_{1}\left (\sum_{i}
x_{i} \alpha_{i}s \right ) \right \}}
\leq \displaystyle{\inf_{s \in S}\left \{\sigma_{2} \left (\sum_{i} x_{i}
\alpha_{i}s \right )\right \}}
=\sigma_{2}^{+'}\left (\displaystyle{\sum_{i}}[x_{i}, \alpha_{i}] \right ),$$ for all
$\displaystyle{\sum_{i}}[x_{i}, \alpha_{i}] \in L $.
Thus, $\sigma_{1}^{+'} \subseteq \sigma_{2}^{+'}$.
Similarly, we can deduce that if $\mu_{1} \subseteq \mu_{2}$, where $\mu_{1}, \mu_{2} \in$
Fh-I$(L)$. Then $\mu_{1}^{+} \subseteq \mu_{2}^{+}$.
We  show that $(\sigma_{1} \oplus \sigma_{2})^{+'}=\sigma_{1}^{+'} \oplus
\sigma_{2}^{+'}$ and $(\sigma_{1} \cap
\sigma_{2})^{+'}=\sigma_{1}^{+'} \cap \sigma_{2}^{+'}$.

Let $\displaystyle{ \sum_{i}[a_{i}, \alpha_{i}]} \in L $. Then
$$
\begin{array}{l}
((\sigma_{1} \oplus \sigma_{2})^{+'})\left (\displaystyle{\sum_{i}}[a_{i},
\alpha_{i}]\right ) \\
=\displaystyle{\inf_{s \in S}\left \{ (\sigma_{1} \oplus
\sigma_{2})\left (\sum_{i} a_{i} \alpha_{i}s \right ) \right \}}\\

 =\displaystyle\inf_{s \in S}\left \{ \sup \left \{\min \left \{  \sigma_{1}\left (\sum_{k} x_{k}
\delta_{k}s), \sigma_{2}(\sum_{j} y_{j} \beta_{j}s \right ) \right \} \ :
\sum_{i}a_{i} \alpha_{i}s=\sum_{k} x_{k} \delta_{k}s+\sum_{j}
y_{j} \beta_{j}s \right \} \right \}
\\
=\displaystyle{\sup \left \{\min \left \{\inf_{s \in S} \left \{ \sigma_{1}\left (\sum_{k} x_{k} \delta_{k}s \right ) \right \} ,
\inf_{s \in S} \left \{ \sigma_{2}\left (\sum_{j}y_{j} \beta_{j}s \right )\right \} \right \} \right \} }\\
=\displaystyle{\sup \left \{ \min \left \{ \sigma^{+'}_{1}\left (\sum_{k}[x_{k}, \delta_{k}]\right ),
 \sigma^{+'}_{2}\left (\sum_{j}[y_{j}, \beta_{j}]\right ) \right \} \right \}}
 \\
 =\left (\sigma^{+'}_{1} \oplus \sigma^{+'}_{2}\right )
 \left (\displaystyle{\sum_{i}}[a_{i}, \alpha_{i}]\right ).
 \end{array}
 $$
 Thus $(\sigma_{1} \oplus \sigma_{2})^{+'}=\sigma^{+'}_{1} \oplus
\sigma^{+'}_{2}$.
Again, we have
$$
\begin{array}{ll}
\displaystyle(\sigma_{1} \cap \sigma_{2})^{+'}\left (\displaystyle \sum_{i}[a_{i},
\alpha_{i}] \right ) &= \displaystyle \inf_{s \in S}\left \{\left (\sigma_{1} \cap \sigma_{2}\right )\left (\sum_{i}
a_{i}\alpha_{i}s \right )\right \}\\
&=\displaystyle{\inf_{s \in
S}\left \{\min \left \{\sigma_{1}\left (\sum_{i}a_{i}\alpha_{i}s \right ),\sigma_{2}\left (\sum_{i}a_{i}\alpha_{i}s \right ) \right \} \right \}}
\\
&=\displaystyle{\min \left \{\inf_{s \in S}\left \{ \sigma_{1}\left (\sum_{i}a_{i}\alpha_{i}s \right ) \right \},
\inf_{s \in S}\left \{ \sigma_{2}\left (\sum_{i}a_{i}\alpha_{i}s \right )\right \}\right  \}}\\
&=\displaystyle{\min \left \{\sigma_{1}^{+'}\left (\sum_{i}[a_{i},\alpha_{i}] \right ),
\sigma_{2}^{+'}\left (\sum_{i}[a_{i},\alpha_{i}]\right )\right \}}
\\
&=\left (\sigma_{1}^{+'} \cap
\sigma_{2}^{+'}\right )\left (\displaystyle{\sum_{i}}[a_{i}, \alpha_{i}]\right ).
\end{array}
$$
So $(\sigma_{1} \cap \sigma_{2})^{+'}=\sigma_{1}^{+'} \cap
\sigma_{2}^{+'}$.
Therefore, the mapping $\sigma \mapsto \sigma^{+'}$ is a lattice
isomorphism.
\end{proof}
Similarly, we can obtain the following theorem.

\begin{Theorem} The lattice of all fuzzy $h$-ideals of $S$ and the lattice of all fuzzy $h$-ideals of $R$ are isomorphic via the mapping $\sigma \mapsto
\sigma^{*'}$, where $\sigma
\in$ Fh-I$(S)$ and $\sigma^{*'} \in$ Fh-I$(R)$.\label{Th:3.9}\end{Theorem}

\begin{Corollary} FLh-I$(L)$, FRh-I$(L)$, FLh-I$(R)$ and FRh-I$(R)$ are complete
lattices.\label{Cor:3.10}\end{Corollary}

\begin{proof} The corollary follows from the above theorems and
the fact that FLh-I$(S)$, FRh-I$(S)$, Fh-I$(S)$ are complete
lattices \cite{Sardar}.
\end{proof}
Now, by routine verification the following lemmas can be obtained.

\begin{Lemma} Let $I$ be an $h$-ideal (left $h$-ideal, right $h$-ideal) of a
$\Gamma$-hemiring $S$ and
$\lambda_{I}$ be the characteristic function of $I$. Then
$(\lambda_{I})^{+'}=\lambda_{(I^{+'})}$.\label{Lemma:3.11}\end{Lemma}

\begin{Lemma} Let $I$ be an $h$-ideal (left $h$-ideal, right $h$-ideal)
of the left operator hemiring $L$ of a $\Gamma$-hemiring $S$ and
$\lambda_{I}$ be the characteristic function of $I$. Then
$(\lambda_{I})^{+}=\lambda_{(I^{+})}$.\label{Lemma:3.12}\end{Lemma}

\begin{Lemma} Let $I$ be an $h$-ideal (left $h$-ideal, right $h$-ideal) of
a $\Gamma$-hemiring $S$ and $\lambda_{I}$ be the characteristic
function of $I$. Then
$(\lambda_{I})^{*'}=\lambda_{(I^{*'})}$.\label{Lemma:3.13}\end{Lemma}

\begin{Lemma} Let $I$ be an $h$-ideal (left $h$-ideal, right $h$-ideal)
of the right operator hemiring $R$ of a $\Gamma$-hemiring $S$ and
$\lambda_{I}$ be the characteristic function of $I$. Then
$(\lambda_{I})^{*}=\lambda_{(I^{*})}$.\label{Lemma:3.14}\end{Lemma}

Now, we revisit the following theorem which is due to Dutta and Sardar \cite{re:Dutta}.

\begin{Theorem} The lattice of all $h$-ideals of $S$ and the lattices of all $h$-ideals of $L$ are isomorphic via the mapping $I \mapsto
I^{+'}$, where $I$ denotes an $h$-ideal of
$S$.\label{Th:3.15}\end{Theorem}

\begin{proof} First, we  show that the mapping $I \mapsto
I^{+'}$ is one-one. Let $I_{1}$ and $I_{2}$ be two $h$-ideals of $S$
such that $I_{1} \neq I_{2}$. Then $\lambda_{I_{1}}$ and
$\lambda_{I_{2}}$ are fuzzy $h$-ideals of $S$, where $\lambda_{I_{1}}$
and $\lambda_{I_{2}}$ are characteristic functions of $I_{1}$ and
$I_{2}$, respectively. Evidently, $\lambda_{I_{1}} \neq
\lambda_{I_{2}}$. Then by Theorem \ref{Th:3.8},
$\lambda_{I_{1}}^{+'} \neq \lambda_{I_{2}}^{+'}$. Hence by Lemma
\ref{Lemma:3.11}, $\lambda_{I_{1}^{+'}} \neq \lambda_{I_{2}^{+'}}$
whence $I_{1}^{+'} \neq I_{2}^{+'}$. Consequently, the mapping $I
\mapsto I^{+'}$ is one-one.
Now, let $J$ be an $h$-ideal of $L$. Then $\lambda_{J}$ is a fuzzy $h$-ideal
of $L$. By Proposition \ref{Prop:3.4} and Theorem \ref{Th:3.8}, $\lambda_{J}^{+}$ is a fuzzy $h$-ideal of $S$.
Now, by Lemma 3.12, $\lambda_{J}^{+}=\lambda_{{J}^{+}}$ and
consequently, $J^{+}$ is an $h$-ideal of $S$. Thus, the mapping is
onto.

Now, let $I_{1},~I_{2}$ be two $h$-ideals of $S$ such that $I_{1} \subseteq
I_{2}$. Then $\lambda_{I_{1}} \subseteq \lambda_{I_{2}}$ and by
Theorem \ref{Th:3.8}, $\lambda_{I_{1}}^{+'} \subseteq
\lambda_{I_{2}}^{+'}$ and by Lemma \ref{Lemma:3.11},
$\lambda_{I_{1}^{+'}} \subseteq \lambda_{I_{2}^{+'}}$ and
consequently $I_{1}^{+'} \subseteq I_{2}^{+'}$. Thus, the mapping
is inclusion preserving. Hence the theorem.
\end{proof}
Similarly, we can revisit the following theorem:

\begin{Theorem} The lattice of all $h$-ideals of $S$ and the lattice of all $h$-ideals of $R$ are isomorphic via the mapping $I \mapsto
I^{*'}$, where $I$ is an $h$-ideal of
$S$.\label{Th:3.16}\end{Theorem}
\begin{Proposition}\label{composition} For any two fuzzy $h$-ideals $\mu$ and $\nu$ of $S$,
$(\mu o_{h}\nu)^{+^{'}}=((\mu)^{+^{'}}o_{h}(\nu)^{+^{'}})$.
\end{Proposition}
\begin{proof} Suppose that $$\displaystyle{\sum_{i}} [x_{i},
e_{i}],\left (\displaystyle{\sum_{i}} [a_{i},
\alpha_{i}]\right )_{j},\displaystyle{\sum_{i}} [z_{i},
\eta_{i}],
\left (\displaystyle{\sum_{i}} [b_{i},
\beta_{i}]\right )_{j},\left (\displaystyle{\sum_{i}} [c_{i},
\gamma_{i}]\right )_{j},\left (\displaystyle{\sum_{i}} [d_{i},
\delta_{i}]\right )_{j}\in L$$
 be
such that $$
\begin{array}{l}
\displaystyle{\sum_{i}} [x_{i},
e_{i}]+\displaystyle{\sum_{j}}\left (\displaystyle{\sum_{i}} [a_{i},
\alpha_{i}]\right )_{j}\left (\displaystyle{\sum_{i}} [c_{i},
\gamma_{i}]\right )_{j}+\displaystyle{\sum_{i}} [z_{i},
\eta_{i}]\\
=\displaystyle{\sum_{j}}\left (\displaystyle{\sum_{i}} [b_{i},
\beta_{i}]\right )_{j}\left (\displaystyle{\sum_{i}} [d_{i},
\delta_{i}]\right )_{j}+\displaystyle{\sum_{i}} [z_{i}, \eta_{i}].
\end{array}
$$
Then
$$
\begin{array}{l}
\left ((\mu)^{+^{'}}o_{h}(\nu)^{+^{'}} \right )\left (\displaystyle{\sum_{i}} [x_{i},
e_{i}]\right )\\

=\sup \left \{ \underset{j}{\min} \left \{ (\mu)^{+^{'}}\left ( \left (\displaystyle{\sum_{i}}
[a_{i}, \alpha_{i}]\right )_{j} \right ),(\nu)^{+^{'}}\left (\left (\displaystyle{\sum_{i}}
[c_{i}, \gamma_{i}]\right )_{j}\right ), \right. \right.\\
\ \ \ \ \ \ \ \ \ \ \left. \left.
 (\mu)^{+^{'}}\left (\left (\displaystyle{\sum_{i}}
[b_{i}, \beta_{i}]\right )_{j}\right ),(\nu)^{+^{'}}\left (\left (\displaystyle{\sum_{i}}
[d_{i},
\delta_{i}]\right )_{j}\right )\right \} \right \}\\

=\sup\left \{ \underset{j}{\min}\left \{ \underset{s\in S
}{\inf}\left \{ \mu \left (\left (\displaystyle{\sum_{i}}a_{i}
\alpha_{i}s\right )_{j}\right ) \right \},\underset{s\in S
}{\inf}\left \{\nu \left (\left (\displaystyle{\sum_{i}}c_{i}
\gamma_{i}s\right )_{j}\right )\right \}, \right. \right.\\
\ \ \ \ \ \ \ \ \ \ \left. \left.
\underset{s\in S
}{\inf}
\left \{\mu
\left (
\left (\displaystyle{\sum_{i}}b_{i}
\beta_{i}s\right )_{j}\right )\right \} ,\underset{s\in S }{\inf}\left \{ \nu \left (\left (\displaystyle{\sum_{i}}d_{i}\delta_{i}s\right )_{j}\right )\right \} \right \} \right \}\\

=\underset{s\in
S}{\inf}\left \{ \sup \left \{ \underset{j}{\min}\left \{\mu \left (\left (\displaystyle{\sum_{i}}a_{i}\alpha_{i}s\right )_{j}\right ),\nu \left (\left (\displaystyle{\sum_{i}}c_{i}
\gamma_{i}s\right )_{j}\right ),\right. \right. \right.\\
\ \ \ \ \ \ \ \ \ \ \left. \left.
\mu \left (\left (\displaystyle{\sum_{i}}b_{i}
\beta_{i}s \right )_{j}\right ),\nu \left (\left (\displaystyle{\sum_{i}}(d_{i}\delta_{i}s\right )_{j}\right )\right \} \right \}  \\

=\underset{s\in S}{\inf}\left \{(\mu o_{h} \nu) \left (\displaystyle{\sum_{i}} x_{i}e_{i}s\right )\right \}\\
=(\mu o_{h}\nu)^{+^{'}}\left (\displaystyle{\sum_{i}} [x_{i},
e_{i}]\right ).
\end{array}
$$
\end{proof}
\begin{Remark} Similarly, we can show that for any two fuzzy $h$-ideals $\mu$ and $\nu$ of $S$,
$(\mu\Gamma_{h}\nu)^{+^{'}}=\mu^{+^{'}}\Gamma_{h}\nu^{+^{'}}$\end{Remark}
\begin{proof} The proof is similar to Proposition \ref{composition} \end{proof}
\begin{Proposition}\label{prime+}
If $\zeta$ is a prime (respectively, semiprime) fuzzy $h$-ideal of $S$, then
$\zeta^{+^{'}}$ (respectively, $\zeta^{*^{'}})$ is a prime (respectively, semiprime) fuzzy $h$-ideal of
$L$ (respectively, $R$).
\end{Proposition}
\begin{proof} Suppose that $\zeta$ is a prime fuzzy $h$-ideal of $S$ and
$\mu^{+^{'}},\nu^{+^{'}}$ be fuzzy $h$-ideals of $L$ such that
$\mu^{+^{'}}\Gamma_{h}\nu^{+^{'}}\subseteq\zeta^{+^{'}}$. Then by using
the above remark we obtain $(\mu
\Gamma_{h}\nu)^{+^{'}}\subseteq\zeta^{+^{'}}$ which implies that $(\mu
\Gamma_{h}\nu)\subseteq\zeta$. Since $\zeta$ is a prime fuzzy $h$-ideal of $S$, then  $\mu\subseteq\zeta$ or
$\nu\subseteq\zeta$,  whence $\mu^{+^{'}}\subseteq\zeta^{+^{'}}$ or
$\nu^{+^{'}}\subseteq\zeta^{+^{'}}$. Therefore, $\zeta^{+^{'}}$ is
a  prime fuzzy $h$-ideal of $L$.
Similarly, we can prove the result for $R$.
Now, for semiprime fuzzy $h$-ideal the proof follows in a similar way.
\end{proof}
\begin{Proposition} If $\zeta$ is a prime (respectively, semiprime) fuzzy $h$-ideal of $L$ (respectively, $R$), then $\zeta^{+}(resp.~\zeta^{*})$ is
a prime (respectively, semiprime) fuzzy $h$-ideal of $S$.
\end{Proposition}
\begin{proof} The proof follows by routine verification.\end{proof}

\begin{Proposition}\label{bi+}
If $\mu$ is a fuzzy $h$-bi-ideal of $S$, then
$\mu^{+^{'}}$ (respectively, $\mu^{*^{'}})$ is a fuzzy $h$-bi-ideal of $L$ (respectively,
$R$).
\end{Proposition}
\begin{proof} Suppose that $\mu$ is a fuzzy $h$-bi-ideal of $S$ and
$\displaystyle{\sum_{i}} [x_{i}, \alpha_{i}],
\displaystyle{\sum_{i}} [y_{i}, \beta_{i}],
\displaystyle{\sum_{i}} [z_{i}, \gamma_{i}]\in$L. Then by Proposition
\ref{$h$-ideal+}, we obtain
$$
\mu^{+^{'}}\left (\left (\displaystyle{\sum_{i}}
[x_{i}, \alpha_{i}]\right )+ \left (\displaystyle{\sum_{i}} [y_{i},
\beta_{i}]\right )\right )
\geq\min\left \{\mu^{+^{'}}\left (\displaystyle \sum_{i} [x_{i},
\alpha_{i}]\right ),\mu^{+^{'}}\left (\displaystyle \sum_{i} [y_{i},
\beta_{i}]\right )\right \}$$ and
$$\mu^{+^{'}}\left (\left (\displaystyle{\sum_{i}} [x_{i},
\alpha_{i}]\right )\left (\displaystyle{\sum_{i}} [y_{i},
\beta_{i}]\right )\right )
\geq\min\left \{\mu^{+^{'}}\left (\displaystyle{\sum_{i}} [x_{i},
\alpha_{i}]\right ),\mu^{+^{'}}\left (\displaystyle{\sum_{i}} [y_{i},
\beta_{i}]\right )\right \}.$$ Now, suppose that
$\displaystyle{\sum_{i}} [x_{i}, e_{i}],\displaystyle{\sum_{i}}
[a_{i}, \alpha_{i}],\displaystyle{\sum_{i}} [z_{i},
\delta_{i}],
\displaystyle{\sum_{i}} [b_{i}, \beta_{i}]\in L$ are such that
$$ \displaystyle{\sum_{i}} [x_{i},
e_{i}]+\displaystyle{\sum_{i}} [a_{i},
\alpha_{i}]+\displaystyle{\sum_{i}} [z_{i},
\delta_{i}]=\displaystyle{\sum_{i}} [b_{i},
\beta_{i}]+\displaystyle{\sum_{i}} [z_{i}, \delta_{i}].$$
Since $\mu^{+^{'}}$ is a fuzzy $h$-ideal of $L$, we have
$$
\mu^{+'}\displaystyle\left (\sum_{i}[x_{i}, e_{i}]\right ) \geq\min \left \{
\mu^{+'}\left (\displaystyle\sum_{i}[a_{i}, \alpha_{i}] \right ),
\mu^{+'}\left (\displaystyle\sum_{i}[b_{i}, \beta_{i}] \right )\right \}.$$
Now, we obtain
 $$
 \begin{array}{ll}
 \mu^{+^{'}}\left (\left (\displaystyle{\sum_{i}} [x_{i}, \alpha_{i}]\right )
\left (\displaystyle{\sum_{i}} [y_{i},
\beta_{i}]\right )\left (\displaystyle{\sum_{i}} [z_{i},
\gamma_{i}]\right )\right )
&
=\mu^{+^{'}}\left (\displaystyle{\sum_{i}}
[x_{i},\alpha_{i}]\displaystyle{\sum_{i}}[y_{i}\beta_{i}z_{i},\gamma_{i}]\right )\\
&
\geq\mu^{+^{'}}\left (\displaystyle{\sum_{i}}
[x_{i},\alpha_{i}]\right ).
\end{array}
$$
Similarly, $$
\begin{array}{ll}
\mu^{+^{'}}\left (\left (\displaystyle{\sum_{i}} [x_{i}, \alpha_{i}]\right )
\left (\displaystyle{\sum_{i}} [y_{i},
\beta_{i}]\right )\left (\displaystyle{\sum_{i}} [z_{i},
\gamma_{i}]\right )\right )
& =\mu^{+^{'}}\left (\displaystyle{\sum_{i}}
[x_{i},\alpha_{i}y_{i}\beta_{i}]\displaystyle{\sum_{i}}[z_{i},\gamma_{i}]\right )\\
&
\geq\mu^{+^{'}}\left (\displaystyle{\sum_{i}}
[z_{i},\gamma_{i}]\right ).
\end{array}
$$
Therefore, we obtain
$$
\begin{array}{l}
\mu^{+^{'}}\left (\left (\displaystyle{\sum_{i}} [x_{i}, \alpha_{i}]\right )
\left (\displaystyle{\sum_{i}} [y_{i},
\beta_{i}] \right )\left (\displaystyle{\sum_{i}} [z_{i},
\gamma_{i}]\right )\right )\\
\geq\min \left \{\mu^{+^{'}}\left (\displaystyle{\sum_{i}} [x_{i},
\alpha_{i}]\right ),\mu^{+^{'}}\left (\displaystyle{\sum_{i}} [z_{i},
\gamma_{i}]\right ) \right \}.
\end{array}
$$
Hence, $\mu^{+^{'}}$ is a fuzzy $h$-bi-ideal of $L$.
Similarly, we can prove the result for $R$.
\end{proof}
\begin{Proposition} If $\mu$ is a fuzzy $h$-bi-ideal of $L$  (respectively, $R$), then $\mu^{+}$(respectively, $ \mu^{*})$ is also a fuzzy
$h$-bi-ideal of $S$.
\end{Proposition}

\begin{Proposition}\label{quasi+}
If $\mu$ is a fuzzy $h$-quasi-ideal of $S$, then
$\mu^{+^{'}}$(respectively, $\mu^{*^{'}})$ is a fuzzy $h$-quasi-ideal of
$L$ (respectively, $R$).
\end{Proposition}
\begin{proof} Suppose that $\mu$ is a fuzzy $h$-quasi-ideal of $S$ and
$\displaystyle{\sum_{i}} [x_{i}, \alpha_{i}],\
\displaystyle{\sum_{i}} [y_{i}, \beta_{i}],\
\displaystyle{\sum_{i}} [z_{i}, \gamma_{i}]\in L$. By Proposition
\ref{$h$-ideal+}, we obtain
$$\mu^{+^{'}}\left (\displaystyle{\sum_{i}}
[x_{i}, \alpha_{i}]+ \displaystyle{\sum_{i}} [y_{i}, \beta_{i}]\right )
\geq\min\left \{\mu^{+^{'}}\left (\displaystyle{\sum_{i}} [x_{i},\alpha_{i}]\right ),\mu^{+^{'}}\left (\displaystyle{\sum_{i}} [y_{i},
\beta_{i}]\right )\right \}.$$ If
$$ \displaystyle{\sum_{i}} [x_{i}, e_{i}]+\displaystyle{\sum_{i}}
[a_{i}, \alpha_{i}]+\displaystyle{\sum_{i}} [z_{i},
\delta_{i}]=\displaystyle{\sum_{i}} [b_{i},
\beta_{i}]+\displaystyle{\sum_{i}} [z_{i}, \delta_{i}],$$
for $
\displaystyle{\sum_{i}} [x_{i}, e_{i}],\displaystyle{\sum_{i}}
[a_{i}, \alpha_{i}],\displaystyle{\sum_{i}} [z_{i},
\delta_{i}], \displaystyle{\sum_{i}} [b_{i}, \beta_{i}]\in L$, then
$$\mu^{+'}\displaystyle{\left (\sum_{i}[x_{i}, e_{i}] \right )}
 \geq\min \left \{
\mu^{+'}\left (\displaystyle\sum_{i}[a_{i}, \alpha_{i}]\right ),
\mu^{+'}\left (\displaystyle\sum_{i}[b_{i}, \beta_{i}]\right )\right \}.$$
Let $\chi_{S}$ be the characteristic function of S. Then by using Proposition \ref{composition} and Theorem
\ref{Th:3.8} we deduce that
$$\left (\mu^{+^{'}} o_{h} \chi_{S}^{+^{'}}\right )\cap \left (\chi_{S}^{+^{'}}
o_{h}\mu^{+^{'}} \right )= (\mu o_{h} \chi_{S})^{+^{'}}\cap(\chi_{S}
o_{h}\mu)^{+^{'}}=((\mu o_{h} \chi_{S})\cap(\chi_{S}
o_{h}\mu))^{+^{'}}\subseteq \mu^{+^{'}}.$$
Thus, $\mu^{+^{'}}$ is a fuzzy h-quasi-ideal of $L$.
Similarly, we can prove the result for $R$.
\end{proof}
\begin{Proposition} If $\mu$ is a fuzzy $h$-quasi-ideal of $L$ (respectively, $R$), then $\mu^{+}$ (respectively, $\mu^{*})$ is also a fuzzy
$h$-quasi-ideal of $S$.
\end{Proposition}

\section{Correspondence of Cartesian Product of Fuzzy $h$-ideals}

Let $\{S_{i}\}_{i \in I}$ be a family of $\Gamma$-hemirings. We define addition
$(+)$ and multiplication $(\cdot)$ on the cartesian product $\displaystyle \prod_{i\in I}S_{i}$ as follows :
$$(x_{i})_{i\in I}+(y_{i})_{i\in I}=(x_{i}+y_{i})_{i\in I},$$
$$(x_{i})_{i\in I}\alpha(y_{i})_{i\in I}=(x_{i}\alpha y_{i})_{i\in
I},$$ for all $(x_{i})_{i\in I},(y_{i})_{i\in I}\in \displaystyle \prod_{i\in I}S_{i}$
and for all $\alpha\in \Gamma$. Then $\displaystyle \prod_{i\in I}S_{i}$ becomes a
 $\Gamma$-hemiring.

\begin{Definition}
\cite{PB} Let $\mu $ and $\sigma $ be two fuzzy subsets of a set
$X$. Then the {\it cartesian product} of $\mu $ and $\sigma $ is defined
by
$$(\mu \times \sigma )(x,y)=\min \{\mu (x),\sigma (y)\},$$
for all $
x,y\in X.$
\end{Definition}

\begin{Definition}
Let $\mu \times \sigma $ be the cartesian product of two fuzzy
subsets $\mu $ and $\sigma $ of $R.$ Then the corresponding
cartesian product $(\mu \times \sigma )^{\ast }$ of $S\times S$ is
defined by
$$(\mu \times \sigma )^{\ast }(x,y)=\underset{\alpha
,\beta \in \Gamma }{\inf } \left \{ (\mu \times \sigma )([\alpha ,x],[\beta
,y])\right \}  ,$$ where $x,y\in S.$
\end{Definition}

\begin{Definition}
Let $\mu \times \sigma $ be the cartesian product of two fuzzy
subsets $\mu $ and $\sigma $ of $S.$ Then the corresponding
cartesian product $(\mu \times \sigma )^{\ast ^{^{\prime }}}$ of
$R\times R$ is defined by $$\left (\mu \times
\sigma \right )^{\ast ^{^{\prime }}}\left (\displaystyle{\sum_{i=1}^{n}}[\alpha_{i} ,x_{i}],\displaystyle{\sum_{j=1}^{m}}[\beta_{j}
,y_{j}]\right )
=\underset{s_{i},s_{j}\in S}{\inf }\left \{ (\mu \times \sigma )\left (\displaystyle{\sum_{i=1}^{n}}s_{i}\alpha_{i}
x_{i},\displaystyle{\sum_{j=1}^{m}}s_{j}\beta_{j} y_{j}\right ) \right \} ,$$ where $\displaystyle{\sum_{i=1}^{n}}[\alpha_{i}
,x_{i}],\ \displaystyle{\sum_{j=1}^{m}}[\beta_{j} ,y_{j}]\in R.$
\end{Definition}

\begin{Definition}
Let $\mu \times \sigma $ be the cartesian product of two fuzzy
subsets $\mu $ and $\sigma $ of $L.$ Then the corresponding
cartesian product $(\mu \times
\sigma )^{+}$ of $S\times S$ is defined by $$(\mu \times \sigma )^{+}(x,y)=%
\underset{\alpha ,\beta \in \Gamma }{\inf }\left \{ (\mu \times \sigma
)([x,\alpha ],[y,\beta ]) \right \} ,$$ where $x,y\in S.$
\end{Definition}

\begin{Definition}
Let $\mu \times \sigma $ be the cartesian product of two fuzzy
subsets $\mu $ and $\sigma $ of $S.$ Then the corresponding
cartesian product $(\mu \times \sigma )^{{+}^{^{\prime }}}$ of
$R\times R$ is defined by
$$(\mu \times \sigma )^{{+}^{^{\prime
}}}\left (\displaystyle{\sum_{i=1}^{n}}[x_{i},\alpha_{i} ],\displaystyle{\sum_{j=1}^{m}}[y_{j},\beta_{j} ]\right )
=\underset{s_{i},s_{j}\in S}{\inf } \left \{ (\mu
\times \sigma )\left (\displaystyle{\sum_{i=1}^{n}}x_{i}\alpha_{i} s_{i},\displaystyle{\sum_{j=1}^{m}}y_{j}\beta_{j} s_{j}\right ) \right \},$$
where $\displaystyle{\sum_{i=1}^{n}}[x_{i},\alpha_{i}
],\displaystyle{\sum_{j=1}^{m}}[y_{j},\beta_{j} ]\in L.$
\end{Definition}
\begin{Proposition}\label{co-product} Let $\mu,\mu^{'},\nu,\nu^{'}$ be four fuzzy $h$-ideals of $S$. Then
$$(\mu\times\mu^{'})\Gamma_{h}(\nu\times\nu^{'})=(\mu\Gamma_{h}\nu)\times(\mu^{'}\Gamma_{h}\nu^{'}).$$
\end{Proposition}
\begin{proof} Let $(x,y)\in S\times S$ be such that
 $$(x,y)+(a,c)\gamma(a^{'},c^{'})+(z,z^{'})=(b,d)\delta(b^{'},d^{'})+(z,z^{'}),$$ where
 $a,c,a^{'},c^{'},z,z^{'},b,d,b^{'},d^{'}\in S$ and $\gamma,\delta\in \Gamma$. Then, $(x,y)+(a\gamma a^{'},c\gamma
 c^{'})+(z,z^{'})=(b\delta b^{'},d\delta d^{'})+(z,z^{'})$ which implies that
  $(x+a\gamma a^{'}+z,y+c\gamma c^{'}+z^{'})=(b\delta b^{'}+z, d\delta d^{'}+z^{'})$. Now, we have
$$
\begin{array}{l}
(\mu\times\mu^{'})\Gamma_{h}(\nu\times\nu^{'})(x,y)\\
=\sup \left \{ \min \left \{(\mu\times\mu^{'})(a,c),(\mu\times\mu^{'})(b,d), (\nu\times\nu^{'})(a^{'},c^{'}),(\nu\times\nu^{'})(b^{'},d^{'})\right \}\right \}\\
=\sup \left \{ \min \left \{\min \left \{\mu(a),\mu^{'}(c) \right \},\min \left \{\mu(b),\mu^{'}(d) \right \}, \min \left \{\nu(a^{'}),\nu^{'}(c^{'}) \right \},\min \left \{\nu(b^{'}),\nu^{'}(d^{'})\right \} \right \}\right \}\\

 =\min \left \{\displaystyle \sup_{x+a\gamma a^{'}+z=b\delta
 b^{'}+z}  \left \{\min \left \{\mu(a),\mu(b),\nu(a^{'}),\nu(b^{'}) \right \}\right \}, \right.
 \\
 \ \ \ \ \ \ \ \ \ \ \ \ \ \ \ \ \ \ \ \ \ \ \ \ \ \left.
 \displaystyle
 \sup_{y+c\gamma c^{'}+z^{'}=d\delta
 d^{'}+z^{'}} \left \{\min \left \{\mu^{'}(c),\mu^{'}(d),\nu^{'}(c^{'}),\nu^{'}(d^{'})\right \}\right \}\right \} \\

 =\min\left \{(\mu\Gamma_{h}\nu)(x),(\mu^{'}\Gamma_{h}\nu^{'})(y)\right \}\\
 =\left \{(\mu\Gamma_{h}\nu)\times(\mu^{'}\Gamma_{h}\nu^{'})\right \}(x,y).
 \end{array}
 $$
 Hence, the proof is completed.
\end{proof}

Note that in this section, we prove the results for the $\Gamma $-hemiring $S$ and its right operator hemiring $R$. Similar results
hold for the $\Gamma $-hemiring $S$ and its left operator hemiring $L$.

\begin{Proposition}
Let $\mu $ and $\sigma $ be two fuzzy subsets of $R$ (respectively, of $L$)[the right (left) operator hemiring of the $\Gamma $-hemiring $S$]. Then $(\mu
\times \sigma )^{\ast }=\mu ^{\ast }\times \sigma ^{\ast }$  (respectively,
$(\mu \times \sigma )^{+}=\mu ^{+}\times \sigma ^{+})$.
\end{Proposition}

\begin{proof} Let $x,y\in S.$ Then
$$
\begin{array}{ll}
(\mu \times \sigma )^{\ast
}(x,y)& =\underset{\alpha ,\beta \in \Gamma }{\inf }\left \{ (\mu \times
\sigma )([\alpha ,x],[\beta ,y]) \right \} \\
& = \underset{\alpha ,\beta \in
\Gamma }{\inf }\left \{ \min \{\mu ([\alpha ,x]),\sigma ([\beta ,y])\} \right \} \\
&
=\min\left \{\underset{\alpha \in \Gamma }{\inf } \left \{ \mu ([\alpha
,x])\right \} ,\underset{\beta \in \Gamma }{\inf } \left \{ \sigma ([\beta ,y])\right \}\right \}\\
&=\min\left \{\mu ^{\ast }(x),\sigma ^{\ast }(y)\right \}\\
&=(\mu ^{\ast }\times
\sigma ^{\ast })(x,y).
\end{array}
$$ Consequently, $(\mu \times \sigma )^{\ast
}=\mu ^{\ast }\times \sigma ^{\ast }.$ Similarly, we can show that
$(\mu \times \sigma )^{+}=\mu ^{+}\times \sigma ^{+}.$
\end{proof}

\begin{Proposition}
Let $\mu $ and $\sigma $ be two fuzzy subsets of $S.$ Then $(\mu
\times \sigma )^{\ast ^{^{\prime }}}=\mu ^{\ast ^{^{\prime
}}}\times \sigma ^{\ast ^{^{\prime }}}$ and $(\mu \times \sigma
)^{+^{^{\prime }}}=\mu ^{+^{^{\prime }}}\times \sigma
^{+^{^{\prime }}}.$
\end{Proposition}

\begin{proof} Let $\displaystyle{\sum_{i=1}^{n}}[\alpha_{i} ,x_{i}],\ \displaystyle{\sum_{j=1}^{m}}[\beta_{j} ,y_{j}]\in
R.$ Then
$$
\begin{array}{l}
(\mu
\times \sigma )^{\ast
^{^{\prime }}}\left (\displaystyle{\sum_{i=1}^{n}}[\alpha_{i} ,x_{i}],\displaystyle{\sum_{j=1}^{m}}[\beta_{j} ,y_{j}]\right )\\
=\underset{s_{i},s_{j}\in S}{\inf }%
\left \{ (\mu \times \sigma )\left (\displaystyle{\sum_{i=1}^{n}}s_{i}\alpha_{i} x_{i},\displaystyle{\sum_{j=1}^{m}}s_{j}\beta_{j}
y_{j}\right ) \right \} \\
=\underset{s_{i},s_{j}\in S}{%
\inf }\left \{ \min \left \{\mu \left (\displaystyle{\sum_{i=1}^{n}}s_{i}\alpha_{i} x_{i}\right ), \left ( \displaystyle{\sum_{j=1}^{m}}\sigma
s_{j}\beta_{j} y_{j} \right )\right \} \right \} \\
=\min\left  \{\underset{%
s_{i}\in S}{\inf } \left \{ \mu \left (\displaystyle{\sum_{i=1}^{n}}s_{i}\alpha_{i} x_{i}\right ) \right \} ,\underset{s_{j}\in S}{\inf
} \left \{ \sigma \displaystyle{\sum_{j=1}^{m}}(s_{j}\beta_{j} y_{j})\right \}\right \} \\
=\min \left \{\mu ^{\ast ^{^{\prime
}}}\left (\displaystyle{\sum_{i=1}^{n}}[\alpha_{i} ,x_{i}]\right ),\sigma ^{\ast ^{^{\prime
}}}\left (\displaystyle{\sum_{j=1}^{m}}[\beta_{j} ,y_{j}] \right )\right \}\\
=(\mu
^{\ast ^{^{\prime }}}\times  \sigma ^{\ast ^{^{\prime }}})\left (\displaystyle{\sum_{i=1}^{n}}[\alpha_{i}
,x_{i}],\displaystyle{\sum_{j=1}^{m}}[\beta_{j} ,y_{j}]\right ).
\end{array}
$$
 Consequently, $(\mu \times \sigma )^{\ast
^{^{\prime }}}=\mu ^{\ast ^{^{\prime }}}\times \sigma ^{\ast
^{^{\prime }}}.$ Similarly, we can show that $(\mu \times \sigma
)^{+^{^{\prime }}}=\mu ^{+^{^{\prime }}}\times \sigma
^{+^{^{\prime }}}.$
\end{proof}
\begin{Proposition}\label{op-product} Let $\mu$ and $\sigma$ be two fuzzy $h$-ideal of $R$ (respectively, $L$).
 Then $\mu^{*}\times\sigma^{*}$ (respectively, $\mu^{+}\times\sigma^{+})$ is a fuzzy $h$-ideal of
$S\times S$.\end{Proposition}
\begin{proof}
Since $\mu$, $\sigma$ are fuzzy $h$-ideals of $R$, by Proposition \ref{Prop:3.6}, $\mu^{*}$, $\sigma^{*}$ are fuzzy
$h$-ideals of $S$ and by Theorem 35 of \cite{Sardar}, we deduce that $\mu^{*}\times\sigma^{*}$ is a fuzzy $h$-ideal of
$S\times S$. In a similar way, we can prove the result for $L$.
\end{proof}

\begin{Proposition}
Let $\mu $ and $\sigma $ be two prime (semiprime) fuzzy
$h$-ideals of $R$ (respectively, of $L$). Then $\mu ^{\ast
}\times \sigma ^{\ast }$ (respectively, $\mu^{+}\times\sigma^{+}$) is prime (semiprime) fuzzy
$h$-ideal of $S\times S.$
\end{Proposition}
\begin{proof} Suppose that $\mu $ and $\sigma $ are two prime fuzzy
$h$-ideals of $R$. By Proposition \ref{op-product}. we see that $\mu^{*}\times\sigma^{*}$ is a fuzzy $h$-ideal of $S\times S$.
In order to show that $\mu^{*}\times\sigma^{*}$ is prime, suppose that $\theta,\theta^{'},\eta,\eta^{'}\in$ Fh-I$(S)$ such that
$(\theta\times\theta^{'})\Gamma_{h}(\eta\times\eta^{'})\subseteq\mu^{*}\times\sigma^{*}$ Then by Proposition
\ref{co-product}, we obtain
$(\theta\Gamma_{h}\eta)\times(\theta^{'}\Gamma_{h}\eta^{'})\subseteq\mu^{*}\times\sigma^{*}$. Therefore,
$(\theta\Gamma_{h}\eta)\subseteq\mu^{*}$ and $(\theta{'}\Gamma_{h}\eta^{'})\subseteq\sigma^{*}$. Hence
$\theta\subseteq\mu^{*}$ or $\eta\subseteq\mu^{*}$ and $\theta{'}\subseteq\sigma^{*}$ or $\eta^{'}\subseteq\sigma^{*}$,
that is, $\theta\times\theta^{'}\subseteq\mu^{*}\times\sigma^{*}$ or
$\eta\times\eta^{'}\subseteq\mu^{*}\times\sigma^{*}$. So, $\mu^{*}\times\sigma^{*}$ is a prime fuzzy $h$-ideal of
$S\times S$.
Similarly, we can prove the result for semiprime fuzzy $h$-ideal and the left operator hemiring $L$.
\end{proof}
By suitable modification of above argument we obtain the following result.
\begin{Proposition}
Let $\mu $ and $\sigma $ be two fuzzy $h$-ideals (prime fuzzy
$h$-ideals, semiprime fuzzy $h$-ideals) of $S.$ Then $\mu ^{\ast
^{^{\prime }}}\times \sigma ^{\ast ^{^{\prime }}}$ (respectively, $\mu ^{+
^{^{\prime }}}\times \sigma ^{+ ^{^{\prime }}}$) is a fuzzy
$h$-ideal (respectively, prime fuzzy $h$-ideals, semiprime fuzzy $h$-ideals) of
$R\times R$ (respectively, $L\times L$).
\end{Proposition}

\begin{Theorem}\label{cores2}
Let $S$ be a $\Gamma $-hemiring with unities and $R$ be its right
operator hemiring . Then there exists an inclusion preserving
bijection $\mu \times \sigma \longmapsto \mu ^{\ast ^{^{\prime
}}}\times \sigma ^{\ast ^{^{\prime }}}$ between the set of all
cartesian product of fuzzy $h$-ideals (respectively, prime fuzzy $h$-ideals, semiprime fuzzy $h$-ideals) of $S$ and the
set of all cartesian product of fuzzy $h$-ideals (respectively, prime fuzzy $h$-ideals, semiprime fuzzy $h$-ideals) of $R,$ where $\mu $ and
$\sigma $ are fuzzy $h$-ideals (respectively, prime fuzzy $h$-ideals, semiprime fuzzy $h$-ideals) of $S.$
\end{Theorem}

\begin{proof} Suppose that $\mu$ and $\sigma$ are fuzzy $h$-ideals of $S$ and $x,y\in S.$ Then
$$
\begin{array}{ll}
\left (\mu ^{\ast ^{^{\prime
}}}\times \sigma ^{\ast ^{^{\prime }}} \right )^{\ast
}(x,y)
&=\underset{\alpha ,\beta \in \Gamma }{\inf }\left \{ (\mu ^{\ast
^{^{\prime }}}\times \sigma ^{\ast ^{^{\prime }}})([\alpha
,x],[\beta ,y])\right \}\\
&=\underset{\alpha ,\beta \in \Gamma }{\inf }\left \{ \min
\left \{\mu ^{\ast ^{^{\prime }}}([\alpha ,x]), \sigma ^{\ast ^{^{\prime
}}}([\beta ,y])\right \}\right \} \\
&=\min \left \{\underset{\alpha \in \Gamma }{\inf }\left \{ \mu
^{\ast ^{^{\prime }}}([\alpha ,x])\right \} ,\underset{\beta \in \Gamma
}{\inf }\left \{ \sigma ^{\ast ^{^{\prime }}}([\beta ,y])\right \} \right \} \\

&=\min \left \{
\underset{\alpha \in \Gamma }{\inf}
\left \{
\underset{s_{1}\in S} {\inf}
\left \{
\mu (s_{1}\alpha x) \right \} \right \} ,
\underset{\beta \in \Gamma } {\inf}
\left \{
\underset{s_{2}\in S}
{\inf }
\left \{
\sigma (s_{2}\beta y)\right \} \right \} \right \} \\
&\geq  \min
\left \{\mu (x),\sigma (y)\right \}\\
&=(\mu \times \sigma )(x,y).
\end{array}
$$ Therefore, $\mu \times \sigma
\subseteq (\mu ^{\ast ^{^{\prime }}}\times \sigma ^{\ast
^{^{\prime }}})^{\ast }.$ Let $[e,\delta ]$ be the strong left unity of
$S.$ Then $e\delta x=x$ and $e\delta y=y$, for all $x,y\in S.$ Now, we have
$$
\begin{array}{ll}
(\mu \times \sigma )(x,y)=\min \{\mu (x),\sigma (y)\}& =\min \{\mu
(e\delta x),\sigma (e\delta y)\}\\

&\geq \min \left \{\underset{\alpha
\in\Gamma }{\inf }\left \{ \underset{s_{1}\in S}{\inf }\left \{ \mu (s_{1}\alpha
x)\right \} \right \} ,\underset{\beta \in \Gamma }{\inf }\left \{ \underset{s_{2}\in
S}{\inf }\left \{ \sigma (s_{2}\beta y)\right \} \right \} \right \}\\

&=\min \left \{\underset{\alpha \in
\Gamma }{\inf }\left \{ \mu ^{\ast ^{^{\prime }}}([\alpha
,x])\right \} ,\underset{\beta \in \Gamma }{\inf }\left \{ \sigma ^{\ast ^{^{\prime
}}}([\beta ,y])\right \}\right \} \\
& =\min \left \{(\mu ^{\ast ^{^{\prime }}})^{\ast
}(x),(\sigma ^{\ast ^{^{\prime }}})^{\ast }(y)\right \}\\
&=((\mu ^{\ast
^{^{\prime }}})^{\ast }\times (\sigma ^{\ast ^{^{\prime }}})^{\ast
})(x,y)\\
&=(\mu ^{\ast ^{^{\prime }}}\times \sigma ^{\ast ^{^{\prime
}}})^{\ast }(x,y).
\end{array}
$$ So, $\mu \times \sigma \supseteq (\mu ^{\ast
^{^{\prime }}}\times \sigma ^{\ast ^{^{\prime }}})^{\ast }.$ Hence,
$\mu \times \sigma =(\mu ^{\ast ^{^{\prime }}}\times \sigma ^{\ast
^{^{\prime }}})^{\ast }.$ Now, let $\mu $ and $\sigma $ be two
fuzzy $h$-ideals of $R.$ Then
 $$
 \begin{array}{l}
 (\mu ^{\ast
}\times \sigma ^{\ast })^{\ast ^{^{\prime }}}\left (\displaystyle{\sum_{i=1}^{n}}[\alpha_{i}
,x_{i}],\displaystyle{\sum_{j=1}^{m}}[\beta_{j}
,y_{j}]\right )
\\

=\underset{s_{i},s_{j}\in S}{\inf }\left \{ (\mu ^{\ast }\times \sigma
^{\ast })\left (\displaystyle{\sum_{i=1}^{n}}s_{i}\alpha_{i} x_{i},\displaystyle{\sum_{j=1}^{m}}s_{j}\beta_{j} y_{j}\right )\right \} \\

 =\underset{s_{i},s_{j}\in
S}{\inf }\left \{ \min \left \{\mu ^{\ast }\left (\displaystyle{\sum_{i=1}^{n}}s_{i}\alpha_{i} x_{i}\right ),\sigma ^{\ast
}\left (\displaystyle{\sum_{j=1}^{m}}s_{j}\beta_{j} y_{j}\right )\right \}\right \} \\

=\min \left \{\underset{s_{i}\in S}{\inf }\left \{ \mu ^{\ast }\left (\displaystyle{\sum_{i=1}^{n}}s_{i}\alpha_{i} x_{i}\right )\right \} ,\underset{s_{j}\in S}{\inf } \left \{ \sigma ^{\ast
}\left (\displaystyle{\sum_{j=1}^{m}}s_{j}\beta_{j} y_{j}\right )\right \}\right \} \\

 =\min \left \{\underset{s_{i}\in S}{\inf
}\left \{ \underset{\gamma \in \Gamma }{\inf }\left \{ \mu \left ([\gamma ,\displaystyle{\sum_{i=1}^{n}}s_{i}\alpha_{i}
x_{i}]\right )\right \} \right \} ,\underset{s_{j}\in S}{\inf }\left \{ \underset{\delta \in \Gamma
}{\inf }\left \{ \sigma \left ([\delta ,\displaystyle{\sum_{j=1}^{m}}s_{j}\beta_{j} y_{j}]\right )\right \} \right \} \right \} \\

=\min
 \left \{\underset{s_{i}\in S}{\inf }\left \{ \underset{\gamma \in \Gamma }{\inf
}\left \{ \mu \left (\displaystyle{\sum_{i=1}^{n}}[\gamma ,s_{i}]\displaystyle{\sum_{i=1}^{n}}[\alpha_{i}
,x_{i}]\right ) \right \} \right \} , \underset{s_{j}\in S}{\inf
}\left \{ \underset{\delta \in \Gamma }{\inf }\left \{ \sigma \left (\displaystyle{\sum_{j=1}^{m}}[\delta ,s_{j}]
\displaystyle{\sum_{j=1}^{m}}[\beta_{j} ,y_{j}]\right )\right \} \right \} \right \} \\
\geq \min \left \{\mu \left (\displaystyle{\sum_{i=1}^{n}}[\alpha_{i} ,x_{j}]\right ),\sigma \left (\displaystyle{\sum_{j=1}^{m}}[\beta_{j}
,y_{j}]\right )\right \}\\
=(\mu \times \sigma )\left (\displaystyle{\sum_{i=1}^{n}}[\alpha_{i} ,x_{i}],\displaystyle{\sum_{j=1}^{m}}[\beta_{j}
,y_{j}]\right ).
\end{array}
$$ Thus, we
obtain $(\mu ^{\ast }\times \sigma ^{\ast })^{\ast ^{^{\prime
}}}\supseteq \mu \times \sigma .$ Let $\displaystyle{\sum_{k=1}^{p}}[\gamma_{k},f_{k}]$ be the right
unity of $S.$ Then
$$
\begin{array}{l}
(\mu \times \sigma )\left (\displaystyle{\sum_{i=1}^{n}}[\alpha_{i}
,x_{i}],\displaystyle{\sum_{j=1}^{m}}[\beta_{j}
,y_{j}]\right )\\
=\min \left \{\mu \left (\displaystyle{\sum_{i=1}^{n}}[\alpha_{i} ,x_{i}]\right ),\sigma \left (\displaystyle{\sum_{j=1}^{m}}[\beta_{j}
,y_{j}]\right )\right \}\\

=\min \left \{\mu
\left (\displaystyle{\sum_{i=1}^{n}}[\alpha_{i} ,x_{i}]\displaystyle{\sum_{k=1}^{p}}[\gamma_{k},f_{k}]\right ), \sigma
\left (\displaystyle{\sum_{j=1}^{m}}[\beta_{j} ,y_{j}]\displaystyle{\sum_{k=1}^{p}}[\gamma_{k},f_{k}]\right )\right \}\\

\geq \min
 \left \{\underset{s_{i}\in S}{\inf }\left \{ \underset{\gamma \in \Gamma
}{\inf }\left \{ \mu \left (\displaystyle{\sum_{i=1}^{n}}[\alpha_{i} ,x_{i}]\displaystyle{\sum_{i=1}^{n}}[\gamma
,s_{i}]\right )\right \} \right \} ,\underset{s_{j}\in
S}{\inf }\left \{  \underset{\delta \in \Gamma }{\inf }\left \{ \sigma\left  (\displaystyle{\sum_{j=1}^{m}}[\beta_{j}
,y_{j}]\displaystyle{\sum_{j=1}^{m}}[\delta ,s_{j}]\right )\right \} \right \} \right \} \\

\geq \min \left \{(\mu ^{\ast })^{\ast ^{^{\prime
}}}\left (\displaystyle{\sum_{i=1}^{n}}[\alpha_{i} ,x_{i}]\right ),(\sigma ^{\ast })^{\ast ^{^{\prime
}}}\left (\displaystyle{\sum_{j=1}^{m}}[\beta_{j}
,y_{j}]\right )\right \}\\

=\left ((\mu ^{\ast })^{\ast ^{^{\prime }}}\times (\sigma ^{\ast
})^{\ast ^{^{\prime }}} \right )\left (\displaystyle{\sum_{i=1}^{n}}[\alpha_{i} ,x_{i}], \displaystyle{\sum_{j=1}^{m}}[\beta_{j}
,y_{j}]\right )\\=(\mu ^{\ast
}\times \sigma ^{\ast })^{\ast ^{^{\prime }}}\left (\displaystyle{\sum_{i=1}^{n}}[\alpha_{i}
,x_{i}],\displaystyle{\sum_{j=1}^{n}}[\beta_{j}
,y_{j}]\right ).
\end{array}
$$
Hence, $(\mu ^{\ast }\times \sigma ^{\ast })^{\ast
^{^{\prime }}}\subseteq \mu \times \sigma .$ Consequently, $(\mu
^{\ast }\times \sigma ^{\ast })^{\ast ^{^{\prime }}}=\mu \times
\sigma .$ Thus, we see that the correspondence $\mu \times \sigma
\longmapsto \mu ^{\ast ^{^{\prime }}}\times \sigma ^{\ast
^{^{\prime }}}$ is a bijection. Now, suppose that $\mu _{1},\mu _{2},\sigma
_{1},\sigma _{2}$ are fuzzy $h$-ideals of $S$ such that $\mu
_{1}\times \sigma _{1}\subseteq \mu_{2}\times \sigma _{2}.$ Then
$$
\begin{array}{l}
\left (\mu _{1}^{\ast ^{^{\prime }}}\times \sigma _{1}^{\ast ^{^{\prime
}}}\right )\left (\displaystyle{\sum_{i=1}^{n}}[\alpha_{i} ,x_{i}],\displaystyle{\sum_{j=1}^{m}}[\beta_{j} ,y_{j}]\right )\\

=\underset{s_{i},s_{j}\in S}{\inf
}\left \{ (\mu _{1}\times \sigma _{1})\left (\displaystyle{\sum_{i=1}^{n}}s_{i}\alpha_{i}
x_{i},\displaystyle{\sum_{j=1}^{m}}s_{j}\beta_{j}
y_{j}\right )\right \} \\

\leq\underset{s_{i},s_{j}\in S}{\inf
}\left \{ (\mu _{2}\times \sigma _{2})\left (\displaystyle{\sum_{i=1}^{n}}s_{i}\alpha_{i}
x_{i},\displaystyle{\sum_{j=1}^{m}}s_{j}\beta_{j}
y_{j}\right )\right \} \\

=\underset{s_{i},s_{j}\in S}{\inf }\left \{ \min \left \{\mu _{2}\left (\displaystyle{\sum_{i=1}^{n}}s_{i}\alpha_{i}
x_{i} \right ),\sigma _{2}\left (\displaystyle{\sum_{j=1}^{m}}s_{j}\beta_{j} y_{j}\right )\right \} \right \} \\

 =\min \left \{\underset{s_{i}\in S}{\inf
}\left \{ \mu _{2}\left (\displaystyle{\sum_{i=1}^{n}}s_{i}\alpha_{i} x_{i}\right )\right \} ,\underset{s_{j}\in S}{\inf }\left \{  \sigma
_{2}\left (\displaystyle{\sum_{j=1}^{m}}s_{j}\beta_{j} y_{j}\right )\right \} \right \} \\
 =\min \left \{\mu _{2}^{\ast ^{^{\prime }}}\left (\displaystyle{\sum_{i=1}^{n}}[\alpha_{i}
,x_{i}]\right ),\sigma _{2}^{\ast ^{^{\prime }}}\left (\displaystyle{\sum_{j=1}^{m}}[\beta_{j} ,y_{j}]\right )\right \}\\
 =\left (\mu_{2}^{\ast ^{^{\prime }}}\times \sigma _{2}^{\ast ^{^{\prime
}}}\right )\left (\displaystyle{\sum_{i=1}^{n}}[\alpha_{i} ,x_{i}],\displaystyle{\sum_{j=1}^{m}}[\beta_{j} ,y_{j}]\right ).
\end{array}
$$ Hence, $\mu
_{1}^{\ast ^{^{\prime
}}}\times \sigma _{1}^{\ast ^{^{\prime }}}\subseteq \mu _{2}^{\ast
^{^{\prime }}}\times \sigma _{2}^{\ast ^{^{\prime }}}.$ Therefore, $\mu
\times \sigma \longmapsto \mu ^{\ast ^{^{\prime }}}\times \sigma
^{\ast ^{^{\prime }}}$ is an inclusion preserving bijection.
Similarly, we can prove the results for prime fuzzy $h$-ideals and semiprime fuzzy $h$-ideals.
\end{proof}
\noindent \textbf{Conclusion:} As a continuation of this paper we will study the  correspondence of $h$-hemiregularity,
$h$-intra-hemiregularity etc in operator hemirings.

\end{document}